\documentclass[11pt, letterpaper]{amsart}

\usepackage{graphicx}
\usepackage[all]{xy}

\newtheorem{theorem}{\noindent{\bf Theorem}}[section]
\newtheorem{definition}[]{\noindent{\bf Definition}}[section]
\newtheorem{proposition}[]{\noindent{\bf Proposition}}[section]

\newtheorem{lemma}[]{\noindent{\bf Lemma}}[section]

\newtheorem{example}[]{\noindent{\bf Example}}[section]
\newtheorem{remark}[]{\noindent{\bf Remark}}[section]

\begin{document}

\title[]{
Isometric immersions of constant curvature one metrics on the 2-sphere with two conical singularities into Euclidean 3-space: A partial solution to the G\'{a}lvez-Hauswirth-Mira problem}
\author {Zhiqiang Wei}


\begin{abstract}
  In this paper, we establish a geometric correspondence between constant curvature one metrics with two conical singularities on $S^{2}$ and isometric immersions into Euclidean 3-space $\mathbb{E}^{3}$. Specifically, we explicitly construct a family of surfaces with constant curvature one, each of which is endowed with two conical singularities. This construction provides a partial solution to an open problem proposed by G\'{a}lvez, Hauswirth, and Mira (Adv in Math. 241(2013) 103-126).

\vspace*{2mm}

\vspace*{2mm}\\
\noindent{\bf 2020 MR Subject Classification:}\hskip3mm 53A05, 53C42.
\thispagestyle{empty}

\end{abstract}
\maketitle

\section{Introduction}
\setcounter{equation}{0}
The isometric immersion problem for Riemannian surfaces into Euclidean 3-space $\mathbb{E}^{3}$ represents a classical challenge in differential geometry, encompassing both the existence and uniqueness (up to rigid motions) of such immersions. Of particular relevance to our work are constant Gaussian curvature one metrics with conical singularities on Riemann surfaces, a subject extensively studied in geometric analysis (see Eremenko's survey \cite{Er21} for historical context). The rich literature on these metrics (see \cite{DFA11,Br88,Xu15,CL91,Er04,Er23,LSX21,LT92,MZ20,MZ22,McO88,MD16,MD19,Tah22,Tr89,Tr91,UY00,WWX24}) provides a strong foundation for our investigation into their isometric immersion properties. \par

In 2013, G\'{a}lvez, Hauswirth, and Mira \cite{JLP13} formulated the following open problem:\par
\textbf{Open problem.} Can one realize any conformal metric of constant curvature one on $S^{2}$ with a finite number of conical singularities as the intrinsic metric of an immersion $K=1$ surface in $\mathbb{E}^{3}$.\par

In their work \cite{JLP13}, they provided examples for the case of metrics with two conical singularities and conical angles in $(0,2\pi)$ using rotationally symmetric surfaces. However, the case of larger conical angles remains unresolved. In this paper, building upon Troyanov's characterization \cite{Tr89} of constant curvature one metrics with two conical singularities, we focus on this specific case. Our work addresses this gap through the following construction, which provides a partial solution to the problem posed by G\'{a}lvez, Hauswirth, and Mira.

\begin{theorem}\label{Mthm}
Let $\alpha>0$ and $0<B<1$ with $\lambda=\frac{\alpha}{B}\in \mathbb{Z}^{+}$. Define coordinate functions:
\begin{equation*}
\begin{cases}
x^{1}(u,\theta)=B\sin u\cos(\lambda\theta),\\
x^{2}(u,\theta)=B\sin u\sin(\lambda\theta),\\
x^{3}(u,\theta)=\int^{u}_{0}\sqrt{1-B^{2}\cos^{2}t}\,dt,
\end{cases}
\end{equation*}
for $(u,\theta)\in (0,\pi)\times [0,2\pi)$. Regard $S^{2}$ as $\mathbb{C}\cup \{\infty\}$ and define $F:\mathbb{C}\setminus\{0\}\rightarrow \mathbb{E}^{3}$ by $z=re^{\sqrt{-1}\theta}\mapsto F(r,\theta)$, where
$$F(r,\theta)=(x^{1}(2\arctan r,\theta),x^{2}(2\arctan r,\theta),x^{3}(2\arctan r,\theta)).$$
This map is an isometric immersion inducing a football metric on $S^{2}$ with conical angles $2\pi\alpha$ at $0$ and $\infty$.
\end{theorem}

\begin{remark}
(i) For $\lambda\neq1$, the immersion breaks rotational symmetry. \par

(ii) When $\alpha\in(0,1)$  and $\lambda=1$ (thereby enforcing $B=\alpha$), the mapping $F$ recovers the rotationally symmetric isometric embedding first introduced in \cite{JLP13}, corresponding to the classical peek sphere described therein.\par

(iii) In the case $\alpha>1$ (thereby forcing $\lambda\geq2$),  the mapping
$F$ produces non-embedded surfaces with self-intersections.
\end{remark}

Let $\rm ds^{2}$ be a constant curvature one metric on $S^{2}$ with two conical singularities. Suppose the two conical angles are both $2\pi\alpha$. If $\rm ds^{2}$ is not a football metric, then $\alpha$ satisfies $2\leq\alpha\in\mathbb{Z}^{+}$(\cite{Tr89}). In this case, an isometric immersion can be explicitly constructed using the following commutative diagram.
$$ \xymatrix{
  S^{2} \ar[rr]^{\tau} \ar[dr]_{F}
                &  &    S^{2} \ar[dl]^{i}    \\
                & \mathbb{E} ^{3}                }$$
Here $\tau$ denotes a branched covering map with just two branching points, $i$ is the natural inclusion and $F=i\circ \tau$ defines the isometric immersion. Then we have the following theorem.

\begin{theorem}\label{Mthm2}
Every constant curvature one metric on $S^{2}$ with two conical singularities admits a regular isometric immersion into $\mathbb{E}^{3}$, smooth away from the singularities.
\end{theorem}

 It is well-known that solving the Gauss and Codazzi equations is essential for addressing isometric immersion problems in
$\mathbb{E}^{3}$. If the underlying surface is simply connected and the Gauss and Codazzi equations admit a solution, the existence of an isometric immersion can be established. However, obtaining an explicit expression for the immersion is generally challenging. Motivated by the form of rotationally symmetric surfaces, we successfully derive the explicit expression presented in \textbf{Theorem} \ref{Mthm}.\par

 This paper is organized as follows. In Sect.2, we review the classification of constant curvature one with two conical singularities on $S^{2}$ and provide the explicit expressions for football metrics on geodesic coordinates. In Sect.3, we present the proof of \textbf{Theorem} \ref{Mthm}. In Sect.4, we elaborate on the derivation of the immersion expressions stated in Sect.3. Finally, in Sect.5, we illustrate several visualizations of footballs in $\mathbb{E}^{3}$ for various parameter values.

\section{Preliminary}

\begin{definition}(\cite{Tr91})
Let $M$ be a Riemann surface,  and let $p\in M$. A conformal metric $\rm ds^{2}$ on $M$ is said to have a conical singularity at $p$ with the singular angle $2\pi\alpha(\alpha>0,\alpha\neq1)$ if, in a neighborhood of $p$,
$${\rm ds^{2}}=e^{2\varphi}|{\rm d}z|^{2},$$
where $z$ is a local complex coordinate defined in the neighborhood of $p$ with $z(p)=0$ and
 $$\varphi-(\alpha -1)\ln|z|$$
is continuous at $0$.
\end{definition}

In \cite{Tr89}, Troyanov classified the constant curvature one metric with two conical singularities on $S^{2}$. Recently, Wei, Wu and Xu \cite{WWX24} provided a new proof of this classification.

\begin{theorem}(\cite{Tr89},\cite{WWX24})
Regard $S^{2}$ as $\mathbb{C}\cup\{\infty\}$. Let ${\rm ds^{2}}$ be a spherical conical metric on $\mathbb{C}\cup\{\infty\}$ with conical singularities at $z=0$ and $z=\infty$,
each with conical angles $2\pi\alpha$ and $2\pi\alpha$ respectively, where $0<\alpha\neq1$. Then, up to a change of coordinate $z\mapsto pz$, where $p\in \mathbb{C}\setminus\{0\}$ is a constant, the metric can be expressed as follows.
\begin{enumerate}
\item If $\alpha\notin~\mathbb{Z}^{+}$, the metric is given by ${\rm ds^{2}}=\frac{4\alpha^{2}|z|^{2(\alpha-1)}}{(1+|z|^{2\alpha})^{2}}|{\rm d}z|^{2}$.
\item If $\alpha\in \mathbb{Z}^{+}$, the metric is given by ${\rm ds^{2}}=\frac{4\alpha^{2}|z|^{2(\alpha-1)}}{(1+|z^{\alpha}+b|^{2})^{2}}|{\rm d}z|^{2}$, where $b\in\mathbb{R}$ is a constant.
\end{enumerate}
\end{theorem}
In \cite{WWX24}, Wei, Wu and Xu introduced the definition of a football metric, which was used to construct constant curvature one metric with finitely many conical singularities on compact Riemann surfaces.
\begin{definition}(\cite{WWX24})
Let ${\rm ds^{2}}$ be a constant curvature one  metric with two conical singularities on a 2-sphere. If the distance between these two singularities is $\pi$, we define ${\rm ds^{2}}$ as a football metric. Furthermore, a 2-sphere equipped with a football metric is referred to as a football.
\end{definition}

By direct calculation, one can easily obtain the following lemma:
\begin{lemma}\label{Lem}
Let ${\rm ds^{2}}$ be a football metric on $S^{2}$ with conical angles $2\pi\alpha$ and $2\pi\alpha$. Then, under the geodesic coordinates, ${\rm ds^{2}}$ can be expressed as
$${\rm ds^{2}}={\rm d}r^{2}+(\alpha\sin r)^{2}{\rm d}\theta^{2},$$
where $(r,\theta)\in(0,\pi)\times[0,2\pi)$.
\end{lemma}

\section{Proof of \textbf{Theorem \ref{Mthm}}}
First, consider the map $\sigma:(0,\pi)\times[0,2\pi)\rightarrow \mathbb{C}\setminus\{0\}=\mathbb{R}^{2}\setminus\{(0,0)\}$, defined by
$$\sigma(u,\theta)=(\tan \frac{u}{2}\cos\theta,\tan \frac{u}{2}\sin\theta).$$
Clearly, $\sigma$ is smooth and bijective, meaning it serves as a parameter transformation. Therefore, to prove \textbf{Theorem} \ref{Mthm}, it suffices to show that the map
$$\vec{r}:(0,\pi)\times[0,2\pi)\rightarrow \mathbb{E}^{3},(u,\theta)\mapsto \vec{r}(u,\theta)=(x^{1}(u,\theta),x^{2}(u,\theta),x^{3}(u,\theta))$$
is an immersion and $\vec{r}~^{*}{\rm ds^{2}_{E}}$ induces a football metric, where ${\rm ds^{2}_{E}}$ denotes the standard Euclidean metric on $\mathbb{E}^{3}$.\par

Since
$$\vec{r}_{u}=(B\cos u\cos(\lambda\theta),B\cos u\sin(\lambda\theta),\sqrt{1-B^{2}\cos^{2}u}),$$
$$\vec{r}_{\theta}=(-B\lambda\sin u\sin(\lambda\theta),B\lambda\sin u\cos(\lambda\theta),0),$$
and $\vec{r}_{u}$ and $\vec{r}_{\theta}$ are linearly independent, $\vec{r}$ is indeed an immersion. Furthermore, a unit normal vector $\vec{n}$ is given by $$\vec{n}=(-\sqrt{1-B^{2}\cos^{2}u}\cos(\lambda\theta),-\sqrt{1-B^{2}\cos^{2}u}\sin(\lambda\theta),B\cos u).$$

Next, we compute the second derivatives:
$$\vec{r}_{uu}=(-B\sin u\cos(\lambda\theta),-B\sin u\sin(\lambda\theta),\frac{B^{2}\cos u\sin u}{\sqrt{1-B^{2}\cos^{2}u}}),$$
$$\vec{r}_{u\theta}=(-B\lambda\sin u\sin(\lambda\theta),B\lambda\sin u\cos(\lambda\theta),0),$$
$$\vec{r}_{\theta\theta}=(-B\lambda^{2}\sin u\cos(\lambda\theta),-B\lambda^{2}\sin u\sin(\lambda\theta),0).$$
The coefficients of the first fundamental form are:
$$E=\vec{r}_{u}\cdot\vec{r}_{u}=1,~~F=\vec{r}_{u}\cdot\vec{r}_{\theta}=0,~~G=\vec{r}_{\theta}\cdot\vec{r}_{\theta}=(B\lambda)^{2}\sin^{2}u.$$
The coefficients of the second fundamental form are:
$$L=\vec{n}\cdot\vec{r}_{uu}=\frac{B\sin u}{\sqrt{1-B^{2}\cos^{2}u}},~~M=\vec{n}\cdot\vec{r}_{u\theta}=0,$$
$$N=\vec{n}\cdot\vec{r}_{\theta\theta}=\lambda^{2}B\sin u\sqrt{1-B^{2}\cos^{2}u}.$$
Thus, the first and second fundamental forms are:
$$I=\vec{r}~^{*}{\rm ds^{2}_{E}}={\rm d}u^{2}+(B\lambda)^{2}\sin^{2}u{\rm d}\theta^{2}={\rm d}u^{2}+\alpha^{2}\sin^{2}u{\rm d}\theta^{2},$$
$$II=\frac{B\sin u}{\sqrt{1-B^{2}\cos^{2}u}}{\rm d}u^{2}+\lambda^{2}B\sin u\sqrt{1-B^{2}\cos^{2}u}{\rm d}\theta^{2}.$$
Moreover, the Gauss curvature is computed to be $1$.\par

From the above proof, we derive the following proposition.
\begin{proposition}
The mean curvature $H$ of the immersion $F$ in Theorem \ref{Mthm} is given by
$$H=\frac{1}{2}(\frac{B\sin u}{\sqrt{1-B^{2}\cos^{2}u}}+\frac{\sqrt{1-B^{2}\cos^{2}u}}{B\sin u})>1,$$
where $u=2\arctan r$.
\end{proposition}

\begin{proof}
First, we observe that $\frac{1}{2}(\frac{B\sin u}{\sqrt{1-B^{2}\cos^{2}u}}+\frac{\sqrt{1-B^{2}\cos^{2}u}}{B\sin u})\geq1$. Equality holds if and only if
 $$\frac{B\sin u}{\sqrt{1-B^{2}\cos^{2}u}}=\frac{\sqrt{1-B^{2}\cos^{2}u}}{B\sin u}$$
 which implies $B=1$. However, this leads to a contradiction.
\end{proof}

\section{Derivation of  the map $\vec{r}$ in Section 3 }
We derive the expression for $\vec{r}$ in Section 3 using the method of rotationally symmetric surfaces, with slight modifications, as follows.\par
Let $f,g:[0,\pi]\rightarrow \mathbb{R}$ be two smooth functions satisfying
\begin{equation*}
\begin{cases}
f(0)=f(\pi)=g(0)=0,\\
(f'(u))^{2}+(g'(u))^{2}=1,\\
g'(u)\neq0,\\
f(u),g(u)>0, \forall u\in (0,\pi).
\end{cases}
\end{equation*}

Suppose $\vec{r}:(0,\pi)\times[0,2\pi)\rightarrow \mathbb{E}^{3}$, defined by
$$\vec{r}(u,\theta)=(x^{1}(u,\theta),x^{2}(u,\theta),x^{3}(u,\theta))$$
is a football with conical angles $2\pi\alpha$ and $2\pi\alpha$ in $\mathbb{E}^{3}$, where
\begin{equation*}
\begin{cases}
x^{1}(u,\theta)=f(u)\cos(\lambda\theta),\\
x^{2}(u,\theta)=f(u)\sin(\lambda\theta),\\
x^{3}(u,\theta)=g(u),
\end{cases}
\end{equation*}
and $\lambda \in\mathbb{Z}^{+}$ is a parameter.\par

Since
$$\vec{r}_{u}=(f'(u)\cos(\lambda\theta),f'(u)\sin(\lambda\theta),g'(u)),$$
$$\vec{r}_{\theta}=(-\lambda f(u)\sin(\lambda\theta),\lambda f(u)\cos(\lambda\theta),0),$$
and $\vec{r}_{u}$ and $\vec{r}_{\theta}$ are linearly independent, $\vec{r}$ is indeed an immersion. Furthermore, a unit normal vector $\vec{n}$ is given by
 $$\vec{n}=(-g'(u)\cos(\lambda\theta),-g'(u)\sin(\lambda\theta),f'(u)).$$

Next, we compute the second derivatives:
$$\vec{r}_{uu}=(f''(u)\cos(\lambda\theta),f''(u)\sin(\lambda\theta),g''(u)),$$
$$\vec{r}_{u\theta}=(-\lambda f'(u)\sin(\lambda\theta),\lambda f'(u)\cos(\lambda\theta),0),$$
$$\vec{r}_{\theta\theta}=(-\lambda^{2} f(u)\cos(\lambda\theta),-\lambda^{2} f(u)\sin(\lambda\theta),0).$$

The coefficients of the first fundamental form are:
$$E=\vec{r}_{u}\cdot\vec{r}_{u}=1,~~~F=\vec{r}_{u}\cdot\vec{r}_{\theta}=0,~~~G=\vec{r}_{\theta}\cdot\vec{r}_{\theta}=(\lambda f(u))^{2}.$$

The coefficients of the second fundamental form are:
$$L=\vec{n}\cdot\vec{r}_{uu}=f'(u)g''(u)-f''(u)g'(u),~~M=\vec{n}\cdot\vec{r}_{u\theta}=0,$$
$$N=\vec{n}\cdot\vec{r}_{\theta\theta}=\lambda^{2}f(u)g'(u).$$

Since $(f'(u))^{2}+(g'(u))^{2}=1$, we have
$$f'(u)f''(u)+g'(u)g''(u)=0,$$
and
$$(f'(u)g''(u)-f''(u)g'(u))g'(u)=-f''(u).$$
Moreover,
$$L=-\frac{f''(u)}{g'(u)}.$$

Thus, the first and second fundamental forms are:
$$I={\rm d}u^{2}+(\lambda f(u))^{2}{\rm d}\theta^{2},$$
$$II=-\frac{f''(u)}{g'(u)}{\rm d}u^{2}+\lambda^{2}f(u)g'(u){\rm d}\theta^{2}.$$

Since the Gauss curvature equals $1$,  we obtain the following equation:
$$f''(u)+f(u)=0.$$

Thus, $f$ can be expressed as
$$f(u)=A\cos u+B\sin u,$$
where $A$ and $B$ are constant.

Noting that $f(0)=f(\pi)=0$, we obtain
$$f(u)=B\sin u,$$
and
$$g(u)=\int_{0}^{u}\sqrt{1-B^{2}\cos^{2}t}dt.$$

Furthermore, the first fundamental form is
$$I={\rm d}u^{2}+(B\lambda)^{2}\sin^{2}u{\rm d}\theta^{2}.$$

Thus, by \textbf{lemma} \ref{Lem}, $\lambda B=\alpha$, i.e, $\lambda=\frac{\alpha}{B}\in\mathbb{Z}^{+}$. \par
Consequently,
\begin{equation*}
\begin{cases}
x^{1}(u,\theta)=B\sin u\cos(\lambda\theta),\\
x^{2}(u,\theta)=B\sin u\sin(\lambda\theta),\\
x^{3}(u,\theta)=\int_{0}^{u}\sqrt{1-B^{2}\cos^{2}t}dt,
\end{cases}
\end{equation*}
where $\lambda B=\alpha$.

\section{Some explicit examples}
In this section, we will present graphical representations of these immersions based on \textbf{Theorem} \ref{Mthm}.

\begin{example}
Let $\alpha=\frac{1}{2}$ and $B=\frac{1}{2}$, which implies $\lambda=1$. The isometric immersion $F$ can then be expressed as (see Figure 1)
$$F(r,\theta)=(x^{1}(2\arctan r,\theta),x^{2}(2\arctan r,\theta),x^{3}(2\arctan r,\theta)),$$
where
\begin{equation*}
\begin{cases}
x^{1}(u,\theta)=\frac{1}{2}\sin u\cos\theta,\\
x^{2}(u,\theta)=\frac{1}{2}\sin u\sin\theta,\\
x^{3}(u,\theta)=\int^{u}_{0}\sqrt{1-\frac{1}{4}\sin^{2}t}dt.
\end{cases}
\end{equation*}

\begin{figure}[htbp]
\centering
\includegraphics[scale=0.40]{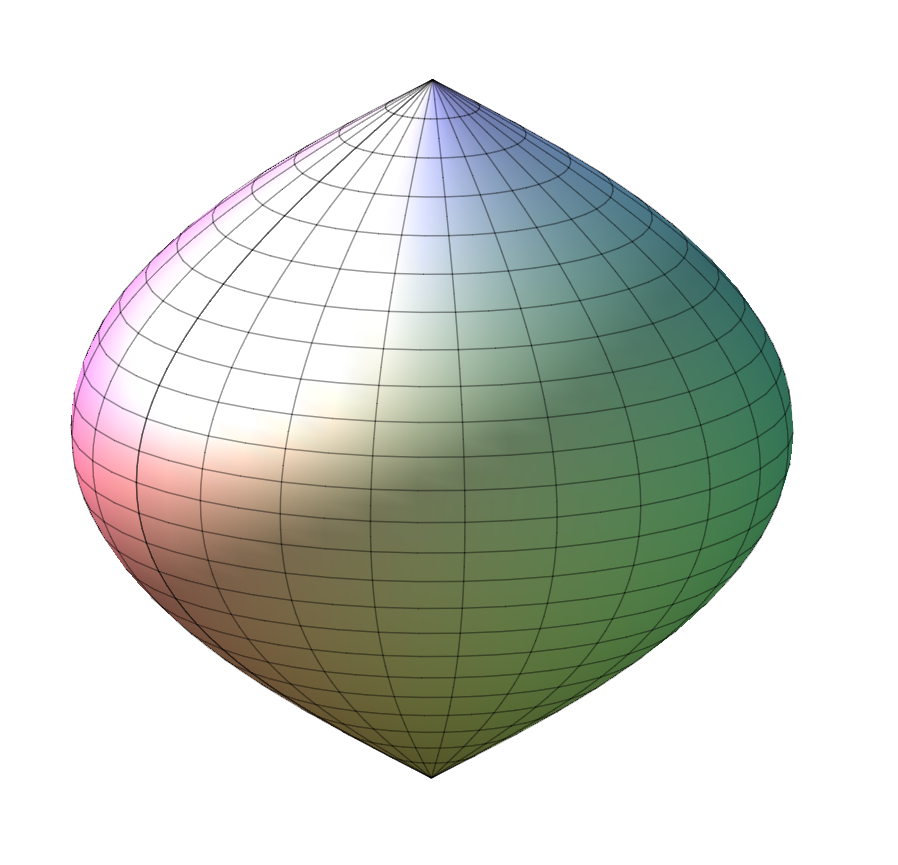}
\caption{\text{ Football of angles $\pi$ and $\pi$.} }
\end{figure}

\begin{figure}[htbp]
\centering
\includegraphics[scale=0.40]{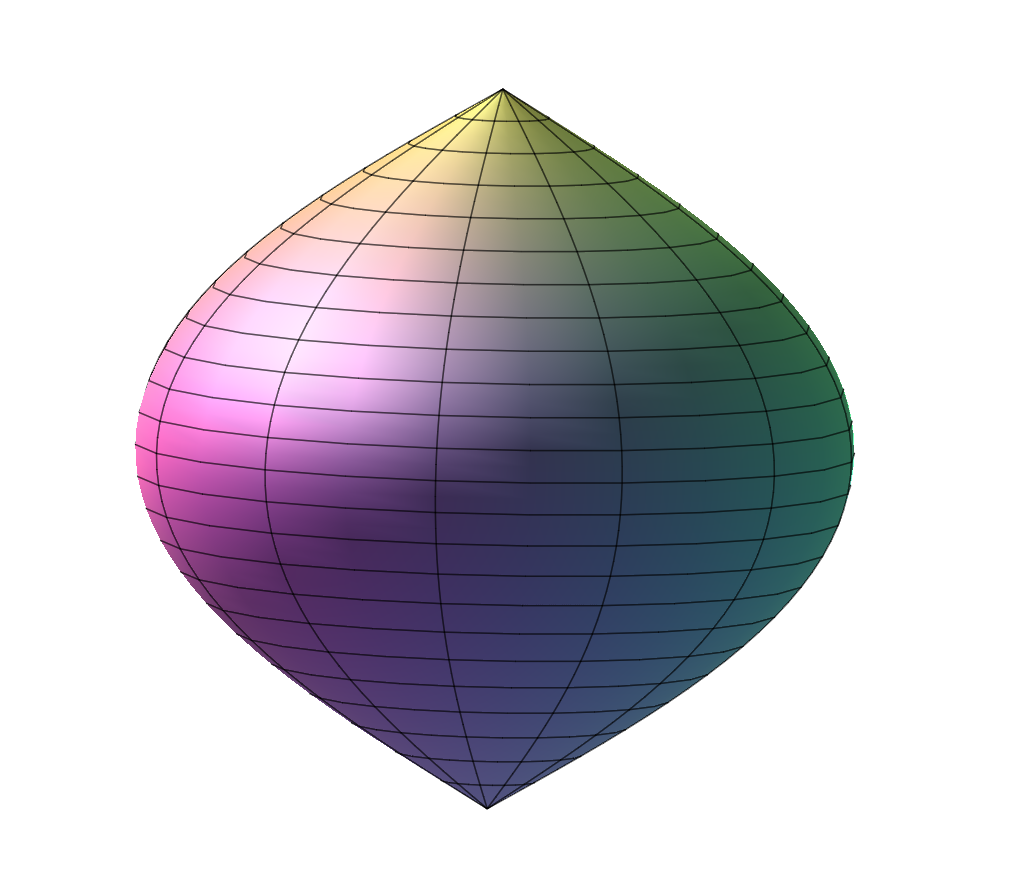}
\caption{\text{ Football of angles $2\pi$ and $2\pi$.} }
\end{figure}
\end{example}

\begin{example}
Let $\alpha=1$ and $B=\frac{1}{2}$, which implies $\lambda=2$. The isometric immersion $F$ can then be expressed as (see Figure 2)
$$F(r,\theta)=(x^{1}(2\arctan r,\theta),x^{2}(2\arctan r,\theta),x^{3}(2\arctan r,\theta)),$$
where
\begin{equation*}
\begin{cases}
x^{1}(u,\theta)=\frac{1}{2}\sin u\cos2\theta,\\
x^{2}(u,\theta)=\frac{1}{2}\sin u\sin2\theta,\\
x^{3}(u,\theta)=\int^{u}_{0}\sqrt{1-\frac{1}{4}\sin^{2}t}dt.
\end{cases}
\end{equation*}

\end{example}

\begin{remark}
Since a singularity of angle $2\pi$ of a constant curvature one metric is smooth, example 2 shows a new shape in of the standard metric on $S^{2}$ in $\mathbb{E}^{3}$.
\end{remark}

\begin{example}
Let $\alpha=2$ and $B=\frac{1}{2}$, which implies $\lambda=4$. The isometric immersion $F$ can then be expressed as (see Figure 3)
$$F(r,\theta)=(x^{1}(2\arctan r,\theta),x^{2}(2\arctan r,\theta),x^{3}(2\arctan r,\theta)),$$
where
\begin{equation*}
\begin{cases}
x^{1}(u,\theta)=\frac{1}{2}\sin u\cos4\theta,\\
x^{2}(u,\theta)=\frac{1}{2}\sin u\sin4\theta,\\
x^{3}(u,\theta)=\int^{u}_{0}\sqrt{1-\frac{1}{4}\sin^{2}t}dt.
\end{cases}
\end{equation*}

\begin{figure}[htbp]
\centering
\includegraphics[scale=0.4]{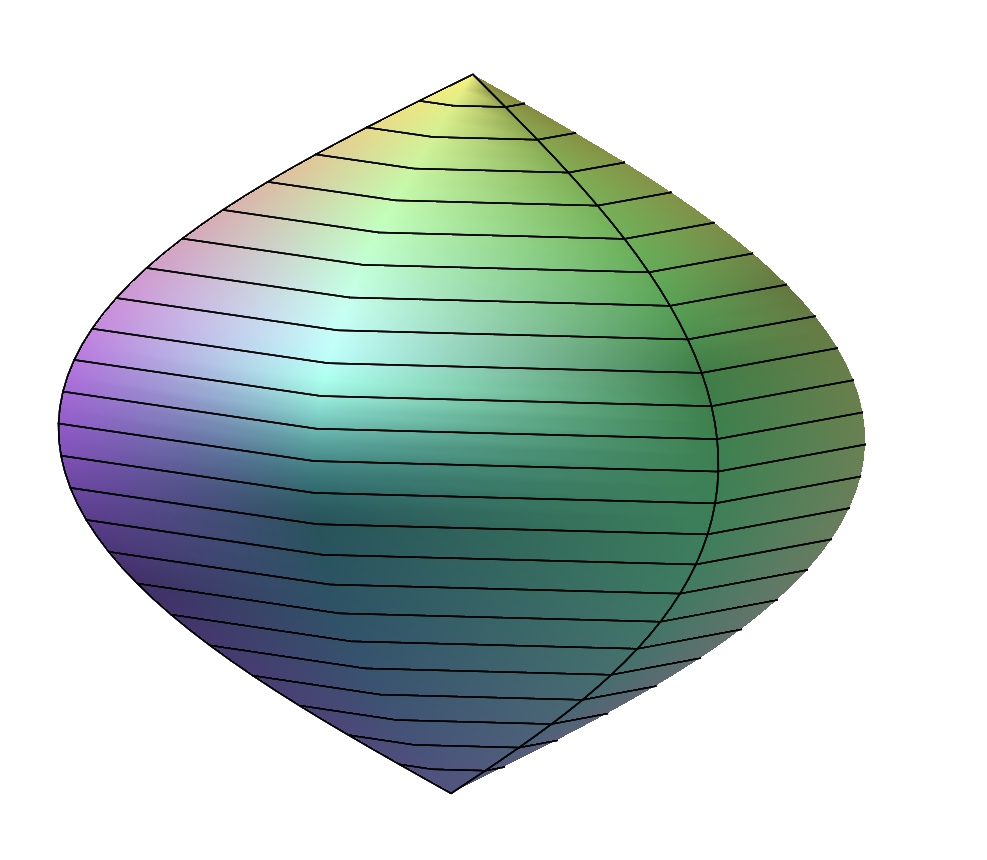}
\caption{\text{ Football of angles $4\pi$ and $4\pi$.} }
\end{figure}

\end{example}

\begin{example}
Let $\alpha=2$ and $B=\frac{1}{4}$, which implies $\lambda=8$. The isometric immersion $F$ can then be expressed as (see Figure 4)
$$F(r,\theta)=(x^{1}(2\arctan r,\theta),x^{2}(2\arctan r,\theta),x^{3}(2\arctan r,\theta)),$$
where
\begin{equation*}
\begin{cases}
x^{1}(u,\theta)=\frac{1}{4}\sin u\cos4\theta,\\
x^{2}(u,\theta)=\frac{1}{4}\sin u\sin4\theta,\\
x^{3}(u,\theta)=\int^{u}_{0}\sqrt{1-\frac{1}{16}\sin^{2}t}dt.
\end{cases}
\end{equation*}

 \begin{figure}[htbp]
\centering
\includegraphics[scale=0.45]{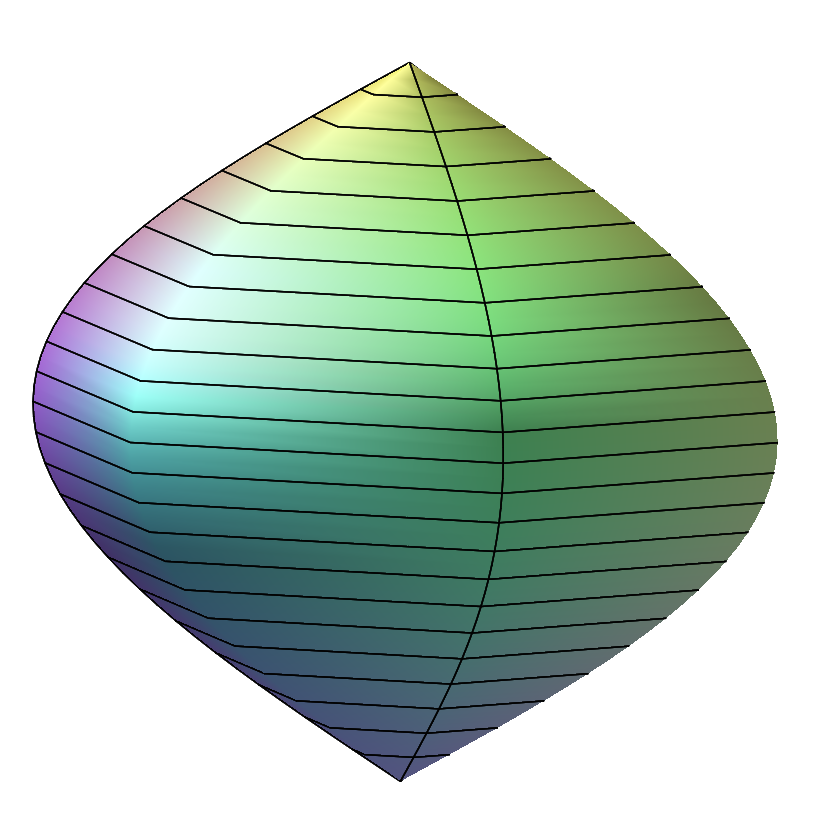}
\caption{\text{ Football of angles $4\pi$ and $4\pi$.} }
\end{figure}

\end{example}

\begin{example}
Let $\alpha=2$ and $B=\frac{1}{8}$, which implies $\lambda=16$. The isometric immersion $F$ can then be expressed as (see Figure 5)
$$F(r,\theta)=(x^{1}(2\arctan r,\theta),x^{2}(2\arctan r,\theta),x^{3}(2\arctan r,\theta)),$$
where
\begin{equation*}
\begin{cases}
x^{1}(u,\theta)=\frac{1}{8}\sin u\cos16\theta,\\
x^{2}(u,\theta)=\frac{1}{8}\sin u\sin16\theta,\\
x^{3}(u,\theta)=\int^{u}_{0}\sqrt{1-\frac{1}{64}\sin^{2}t}dt.
\end{cases}
\end{equation*}

 \begin{figure}[htbp]
\centering
\includegraphics[scale=0.45]{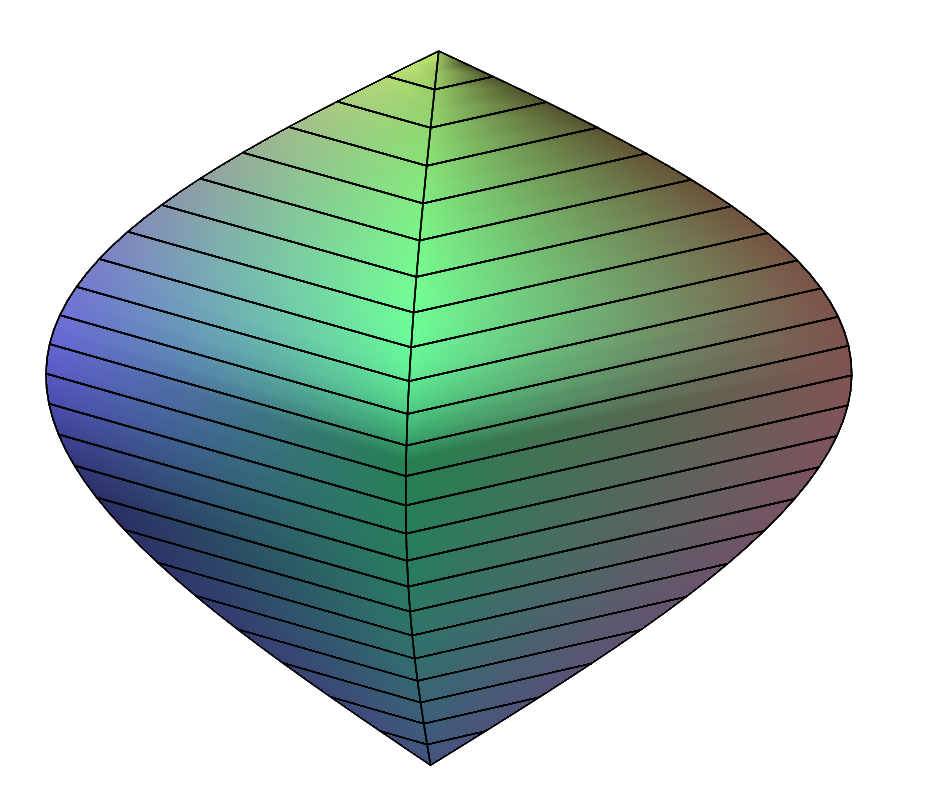}
\caption{\text{ Football of angles $4\pi$ and $4\pi$.} }
\end{figure}

\end{example}

Finally, we will give an example of constant curvature 1 metric on $S^{2}$ with two conical singularities by using the diagram:

$$ \xymatrix{
  S^{2} \ar[rr]^{\tau} \ar[dr]_{F}
                &  &    S^{2} \ar[dl]^{i}    \\
                & \mathbb{E} ^{3}                }$$
\begin{example}
Let $\tau:\mathbb{C}\rightarrow \mathbb{C}$ be a holomorphic function defined by $\tau(z)=z^{2}+1$, which implies $\alpha=2$. The isometric immersion $F$ can then be expressed as (see Figure 6)
$$F(r,\theta)=(x^{1}(r,\theta),x^{2}(r,\theta),x^{3}(r,\theta)),$$
where
\begin{equation*}
\begin{cases}
x^{1}(r,\theta)=\frac{2(r^{2}\cos(2\theta)+1)}{1+(r^{2}\cos(2\theta)+1)^{2}+r^{4}\sin^{2}(2\theta)},\\
x^{2}(r,\theta)=\frac{2r^{2}\sin(2\theta)}{1+(r^{2}\cos(2\theta)+1)^{2}+r^{4}\sin^{2}(2\theta)},\\
x^{3}(r,\theta)=\frac{(r^{2}\cos(2\theta)+1)^{2}+r^{4}\sin^{2}(2\theta)-1}{1+(r^{2}\cos(2\theta)+1)^{2}+r^{4}\sin^{2}(2\theta)},
\end{cases}
\end{equation*}
where $z=re^{i\theta}$.

 \begin{figure}[htbp]
\centering
\includegraphics[scale=0.45]{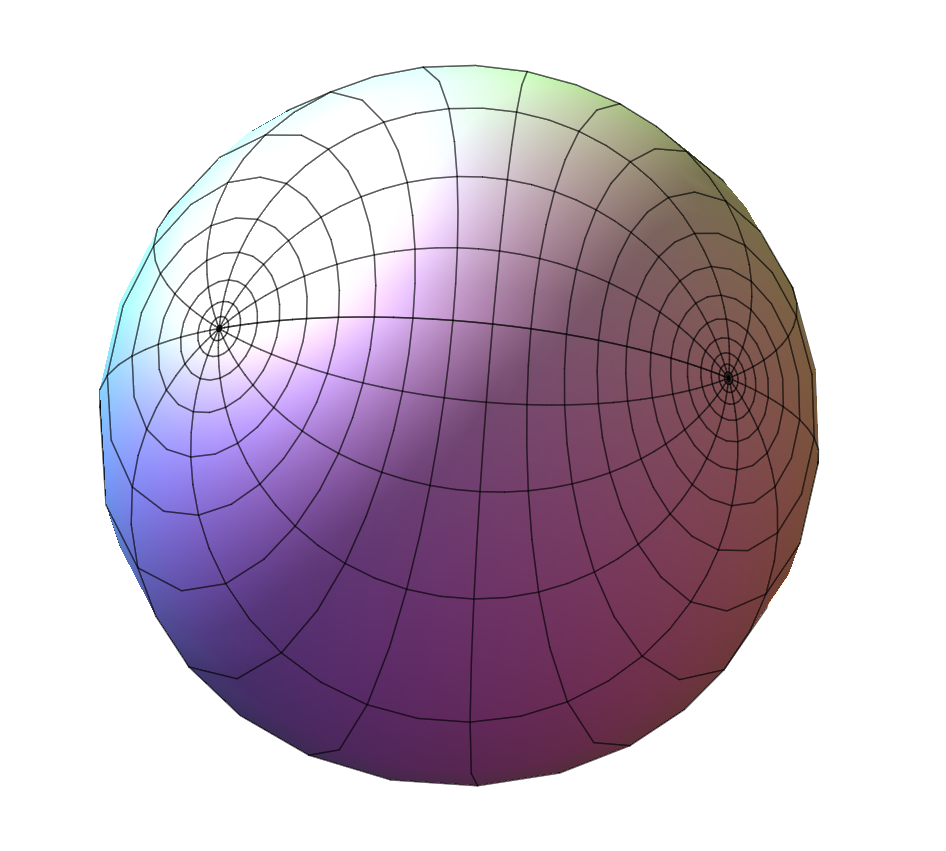}
\caption{\text{ Two conical angles are $4\pi$ and $4\pi$.} }
\end{figure}

\end{example}


\textbf{Declarations}

\textbf{Data Availability Statement}  This manuscript has no associated data.

\textbf{Competing interests} On behalf of all authors, the corresponding author states that there is no conflict of interest.

\noindent
Zhiqiang Wei\\
School of Mathematics and Statistics, Henan University, Kaifeng 475004 P. R. China\\
Center for Applied Mathematics of Henan Province, Henan University, Zheng zhou 450046 P. R. China\\
Email: weizhiqiang15@mails.ucas.edu.cn. ~or~10100123@vip.henu.edu.cn.

\end{document}